\documentclass[11pt]{amsart}
\usepackage[T1]{fontenc}
\usepackage[utf8]{inputenc}
\usepackage{lmodern}
\usepackage{amsmath,amssymb,amsthm}
\usepackage[letterpaper,margin=28mm]{geometry}
\usepackage{microtype}
\usepackage[dvipsnames]{xcolor}
\definecolor{annotationblue}{RGB}{0,0,200}

\usepackage[hidelinks]{hyperref}

\newtheorem{proposition}{Proposition}
\newtheorem{theorem}{Theorem}
\title{Quotients of $L_1$ by subsequences of the Haar system have cotype 2}
\author[S.J.~Dilworth]{Stephen Dilworth}
\address{S. J. Dilworth, Department of Mathematics, University of South Carolina, Columbia, SC 29208. USA}
\email{dilworth@math.sc.edu}

\author[D.~Kutzarova]{Denka Kutzarova}
\address{Denka Kutzarova, Department of Mathematics, University of Illinois Urbana-Champaign,
Urbana, IL 61807, USA; Institute of Mathematics and Informatics, Bulgarian Academy of Sciences, Sofia, Bulgaria}
\email{denka@illinois.edu}

\author[M.~Ostrovskii]{Mikhail Ostrovskii}
\address{Mikhail Ostrovskii, Department of
Mathematics and Computer Science, St. John's University, 8000 Utopia Parkway, Queens, NY 11439, USA} 
\email{ostrovsm@stjohns.edu}

\author[Th.~Speckhofer]{Thomas Speckhofer}%
\address{Thomas Speckhofer, Department of Mathematics, Texas A\&M University, College Station, TX~77843, USA}
\email{speckhofer@tamu.edu}

\date\today



\thanks{The first author was supported by Simons Foundation Collaboration Grant No. 849142.  The second author was supported by Simons Foundation Collaboration Grant No. 636954. The fourth author was supported by the Austrian Science Fund~(FWF), project 10.55776/J5020.}

\newcommand{\tc}{{\rm TC}\hskip0.02cm}
\def \ep {\varepsilon}
\def \R {\mathbb{R}}

\def \N {\mathbb{N}}

\begin{document}

\begin{abstract}
     Let $Y$ be the closed linear span of a subsequence of the Haar system in $L_1$. We prove that the quotient space $L_1/Y$ has cotype~$2$. This result implies that transportation cost spaces on diamond graphs have cotype $2$.
\end{abstract}

\maketitle

\noindent\textbf{AI usage statement.}
The proofs in this paper were generated by OpenAI’s ChatGPT 6.0 Astra during a chat initiated by the authors. The model was prompted to determine the above cotype using different approaches suggested by the authors. In particular, the authors suggested following the the proof of Bourgain's Theorem on~$L_1/H^1$, as presented in Pisier~\cite{MR829919}. The authors have checked the mathematical correctness of results and proofs and edited them for clarity, presentation, and context.

\section{Introduction}

The main reason for our interest in the quotient spaces of $L_1$ by subsequences of Haar functions is that some of them are isometric to transportation cost spaces on the well-known diamond graphs (see \cite{MR4098600}). We denote a transportation cost space on a metric space $X$ by $\tc(X)$. The study of the cotype of transportation cost spaces was initiated by Bourgain \cite{MR880292}. He proved that $\{\tc(H_n)\}$, where $\{H_n\}_{n=1}^\infty$ are Hamming cubes, have no cotype in the sense that for each $d\in\N$ and $\ep>0$ there is $n\in\N$ such that $\ell_\infty^d$ is $(1+\ep)$-isometric to a subspace of $\tc(H_n)$. Bourgain's motivation for this result was that it showed that the suggestion to introduce the cotype of a metric space $X$ as the (Banach-space-theoretical) cotype of $\tc(X)$ does not seem suitable because, in this sense, the space $\ell_1$ has no cotype. Bourgain also noted that determining the cotype of $\tc(\R^2)$ is significantly more difficult (this problem remains open; see the discussion in Naor \cite{MR3966745}). Note that it would suffice to solve a version of this problem for square plane grids.

The interest in the result of Bourgain on $\tc(H_n)$ increased after computations with transportation cost became important for Computer Vision \cite{IT03}, where the precision of the results is sometimes sacrificed for the sake of speed. One way to achieve fast approximate computations of transportation cost is based on low-distortion embeddings of the corresponding spaces $\tc(X)$ into $L_1$. Bourgain's result on $\{\tc(H_n)\}_{n=1}^\infty$ shows that there is a cotype obstruction for their low-distortion embeddings into $L_1$. It became important to determine for which of the well-known metric spaces $X$, cotype is an obstruction for the low-distortion embedding of $\tc(X)$
into $L_1$. 

Our results show that diamond graphs $\{D_n\}_{n=1}^\infty$ provide the first known example of a family of metric spaces for which the $L_1$-distortion grows indefinitely as $n\to\infty$ (see \cite{BGS23}), and even this growth has the largest possible order \cite{GO26}, but the cotype is not an obstacle for their embeddings into $L_1$.

The study of the cotype properties of quotients of $L_1$ has a long history. It started with the results of Kislyakov \cite{Kis76} and Pisier \cite{Pis78}. It was continued with contributions of Bourgain \cite{Bou84} and Bourgain-Davis \cite{MR825718}. In this paper, we use a version of the Bourgain-Davis result (Theorem \ref{thm:bourgain-davis}).

Another starting point for this study was the recent result of \cite{speckhofer:even-levels}.
\bigskip

Let $(h_I)_{I\in \mathcal{D}}$ denote the Haar system, where $\mathcal{D}$ is the set of all dyadic intervals in~$[0,1]$. Our main result is the following.
\begin{theorem}\label{thm:main-intro}
  Let $\mathcal{A}$ be a subcollection of the dyadic intervals~$\mathcal{D}$, and let $Y$ denote the closed linear span of $\{ h_I : I\in \mathcal{A} \}$ in $L_1$. Then the quotient space $L_1/Y$ has cotype~$2$.
\end{theorem}

Before proving Theorem~\ref{thm:main-intro}, we establish some notation.
Let $\mathbb{N}_0 = \mathbb{N}\cup \{ 0 \}$. We denote the collection of dyadic intervals in $[0,1)$ by
\begin{equation*}
\mathcal{D} = \Big\{\Big[\frac{i-1}{2^n},\frac{i}{2^n}\Big):1\leq i\leq 2^n,\ n\in \mathbb{N}_0\Big\}.
\end{equation*}
Moreover, for $N\in \mathbb{N}_0$, we define
\begin{equation*}
  \mathcal{D}_N = \{ I\in \mathcal{D} : |I| \ge 2^{-N} \}.
\end{equation*}
 For each dyadic interval $I\in \mathcal{D}$, let~$I^+$ denote the left half of~$I$ and~$I^-$ its right half (both are again elements of $\mathcal{D}$).
The Haar system $(h_I)_{I\in\mathcal{D}}$ is defined as
\begin{equation*}
  h_I
  = \mathbf{1}_{I^+} - \mathbf{1}_{I^-},
  \qquad I\in\mathcal{D}.
\end{equation*}
Together with the constant function $\mathbf{1}_{[0,1)}$, the Haar system in its usual lexicographic order is a monotone Schauder basis
of $L_1$. For $N\in \mathbb{N}_0$, we put $L_1^N = \operatorname{span}(\{ \mathbf{1}_{[0,1)} \}\cup \{ h_I : I\in \mathcal{D}_N \})$.
We use $\int f$ and $\int f\,dt$ for integrals with respect to the Lebesgue measure on $[0,1]$.

By density, our main result (Theorem~\ref{thm:main-intro}) follows from the following finite-dimensional theorem.

\begin{theorem}\label{thm:main}
There is an absolute constant $C$ such that, for every $N\geq0$, every
collection $\mathcal{A}\subset \mathcal{D}_N$, every $n\geq1$, and all
$x_1,\ldots,x_n\in L_1^N$,
\begin{equation}\label{eq:main}
 \biggl(\sum_{j=1}^n\left\|x_j+Y\right\|_{L_1/Y}^2\biggr)^{1/2}
 \leq C\mathbb{E}_\varepsilon\biggl\|\sum_{j=1}^n\varepsilon_jx_j+Y\biggr\|_{L_1/Y},
\end{equation}
where $Y = \operatorname{span}\{ h_I : I\in \mathcal{A} \}$ and $\mathbb{E}_\varepsilon$ is the average over all sign vectors $\varepsilon\in \{ -1,1 \}^n$.
\end{theorem}

\section{A weighted orthogonal projection}

To prove Theorem~\ref{thm:main}, we choose representatives $z_j$ of the classes $x_j+Y$ using a weighted orthogonal projection. Thus, for every signed sum $\sum_{j=1}^n\varepsilon_jx_j$ in~\eqref{eq:main}, we obtain a representative by the same linear rule, independent of~$\varepsilon$. After a change of measure, this projection becomes a martingale transform, and thus, we can apply the Bourgain--Davis inequality (Theorem~\ref{thm:bourgain-davis}) to get a mixed-norm inequality (Proposition~\ref{prop:weighted}). This inequality will be used in Section~\ref{sec:proof-of-main-result} to complete the proof of~\eqref{eq:main}.

Fix $N\in\mathbb{N}_0$, a collection $\mathcal{A}\subset\mathcal{D}_N$, and let $Y = \operatorname{span}\{ h_I : I\in \mathcal{A} \}$. For a strictly positive function $w\in L_1^N$, we will use
the inner product
\begin{equation}\label{eq:inner-product}
 \langle f,g\rangle_w=\int fg/w.
\end{equation}
Let $P_w$ be the orthogonal projection from $L_1^N$ onto $Y^{\perp_w}$ with respect to the above inner product, so
\begin{equation}\label{eq:P}
 P_w\text{ is linear},\qquad P_wy=0 \text{ for }y\in Y,\qquad P_wf-f\in Y \text{ for }f\in L_1^N.
\end{equation}

\begin{proposition}\label{prop:weighted}
For every $n\geq1$ and functions $f_1,\ldots,f_n\in L_1^N$,
\begin{equation}\label{eq:weighted}
 \int\biggl(\frac1n\sum_{i=1}^n|P_wf_i|^{1/2}\biggr)^4\,\frac{dt}{w}
 \leq B^2\int\biggl(\frac1n\sum_{i=1}^n|f_i|\biggr)^2\,\frac{dt}{w},
\end{equation}
with the absolute constant $B = C_{1/2}$ from Theorem~\ref{thm:bourgain-davis}.
\end{proposition}

\begin{proof}
  For a dyadic interval $I\in \mathcal{D}_{N+1}$, let $Y_I = \operatorname{span}\{ h_K : K\in \mathcal{A},\ K\subset I \}$.
  Moreover, let
  \begin{equation*}
    b_I = \mathbf{1}_I - P_{Y_I} \mathbf{1}_I,
  \end{equation*}
  where $P_{Y_I}$ denotes the orthogonal projection onto~$Y_I$ with respect
  to the inner product defined in~\eqref{eq:inner-product}.
  Thus, $b_I - \mathbf{1}_I \in Y_I$ and $\langle b_I, y \rangle_w = 0$ for all $y\in Y_I$.
  At level $N+1$, where $|I|=2^{-(N+1)}$, no Haar function $h_K$ with $K\in \mathcal{A}$ is supported on $I$. Thus $Y_I=\{0\}$ and $b_I=\mathbf{1}_I$.

  For $I\in \mathcal{D}_N$, we claim that
  \begin{alignat}{2}
    b_I &= b_{I^+}+b_{I^-}
        &\qquad&\text{if }I\in \mathcal{D}_N\setminus \mathcal{A},\label{eq:b-A-compl}\\
    b_I &= \frac{2\beta}{\alpha+\beta}b_{I^+}
           +\frac{2\alpha}{\alpha+\beta}b_{I^-}
        &\qquad&\text{if }I\in\mathcal{A},\label{eq:b-A}
  \end{alignat}
  where $\alpha = \|b_{I^+}\|_w^2$ and $\beta = \|b_{I^-}\|_w^2$.
  This is proved by induction, starting at level~$N$ and using the base case at level~$N+1$.
  In particular, we will show that $b_I > 0$ on $I$ (which is clear at level~$N+1$).
  Hence, fix $I\in \mathcal{D}_N$ and assume that $b_{I^+} > 0$ on $I^+$ and $b_{I^-} > 0$ on $I^-$, so $\alpha,\beta > 0$.
  
  If $I\notin \mathcal{A}$, then $Y_I = Y_{I^+}\oplus_w Y_{I^-}$, and the summands are orthogonal under~\eqref{eq:inner-product} because their supports are disjoint.
  Thus, we obtain~\eqref{eq:b-A-compl}:
  \begin{equation*}
    b_I=(\mathbf{1}_{I^+}+\mathbf{1}_{I^-})-(P_{Y_{I^+}}\mathbf{1}_{I^+}+P_{Y_{I^-}}\mathbf{1}_{I^-}) =b_{I^+}+b_{I^-}.
  \end{equation*}
  
  If $I\in \mathcal{A}$, put $D_I = Y_{I^+} \oplus_w Y_{I^-}$ and $k_I=h_I-P_{D_I}h_I$. By linearity,
  \begin{equation}\label{eq:kI-A}
    k_I
 =(\mathbf{1}_{I^+}-P_{Y_{I^+}}\mathbf{1}_{I^+})-(\mathbf{1}_{I^-}-P_{Y_{I^-}}\mathbf{1}_{I^-})=b_{I^+}-b_{I^-}.
\end{equation}
Since $h_I-k_I\in D_I$ and $k_I\perp_w D_I$, we can write $Y_I=D_I\oplus_w\operatorname{span}\{k_I\}$.
Thus, to compute~$b_I$, we can start with $\mathbf{1}_I$, subtract its orthogonal projection onto $D_I$, leaving $b_{I^+} + b_{I^-}$, and then subtract its orthogonal projection onto the one-dimensional space $\operatorname{span}\{ k_I \}$. Since $b_{I^+}$ and~$b_{I^-}$ have disjoint supports, we see that 
$\langle b_{I^+}+b_{I^-},k_I\rangle_w=\alpha-\beta$ and $\|k_I\|_w^2=\alpha+\beta$. Hence
\begin{equation*}
 b_I=b_{I^+}+b_{I^-}-\frac{\alpha-\beta}{\alpha+\beta}(b_{I^+}-b_{I^-}),
\end{equation*}
which completes the proof of~\eqref{eq:b-A}.

In both cases, $b_I$ is strictly positive on $I$ because its restrictions to $I^+$ and $I^-$ are positive multiples of $b_{I^+}$ and $b_{I^-}$, respectively.
Moreover, it follows by induction that for each $I\in \mathcal{D}_{N+1}$, $(k_J : J\in \mathcal{A},\ J\subset I)$ is an orthogonal basis of~$Y_I$:
At level $N + 1$, the space is zero and the family is empty. For $I\in \mathcal{D}_N\setminus \mathcal{A}$, combine the bases of $Y_{I^+}$ and $Y_{I^-}$. For $I\in \mathcal{A}$, this follows from the above orthogonal decomposition
\begin{equation*}
  Y_I = D_I \oplus_w \operatorname{span}\{ k_I \} = Y_{I^+}\oplus_w Y_{I^-}\oplus_w \operatorname{span}\{ k_I \}.
\end{equation*}
In particular, for $I = [0,1)$, we obtain that $(k_J)_{J\in \mathcal{A}}$ is an orthogonal basis of $Y$.

Now put $b=b_{[0,1)}=P_w\mathbf{1}_{[0,1)}$. By the above observation, the restriction~$b|_I$ is a positive scalar multiple of $b_I$ for every $I\in \mathcal{D}_{N+1}$.
Define the positive finite measure
\begin{equation*}
 d\nu=(b^2/w)\,dt.
\end{equation*}

For each interval $I\in \mathcal{A}$, we claim that the function $k_I/b$ is a nonzero
multiple of the Haar function at $I$ for the measure $\nu$, which is defined as
\begin{equation*}
 h_I^\nu=\frac{\mathbf{1}_{I^+}}{\nu(I^+)}-\frac{\mathbf{1}_{I^-}}{\nu(I^-)}.
\end{equation*}
Indeed, write $b|_{I^+}=t_{I^+}b_{I^+}$ and $b|_{I^-}=t_{I^-}b_{I^-}$, with $t_{I^+},t_{I^-}>0$.
Then, by~\eqref{eq:kI-A},
\begin{equation*}
 \frac{k_I}{b}=\frac{\mathbf{1}_{I^+}}{t_{I^+}}-\frac{\mathbf{1}_{I^-}}{t_{I^-}}.
\end{equation*}
Since $b\perp_w Y$ and $k_I\in Y$, we have
$\int(k_I/b)\,d\nu=\int k_Ib/w=0$. Its zero integral implies that $k_I/b$ is indeed a nonzero multiple of $h_I^{\nu}$.

Together with $\mathbf{1}_{[0,1)}$, the functions $h_I^{\nu}$, $I\in \mathcal{D}_N$, form an orthogonal basis of $L_1^N$ equipped with the inner product $\int gh\,d\nu$.
Let $T_\nu$ be the projection that maps $h_I^\nu$ to zero for $I\in\mathcal{A}$ and fixes the remaining basis functions, including $\mathbf{1}_{[0,1)}$.
Thus, $T_\nu$ is the orthogonal projection onto $(Y^\nu)^\perp$, where
\begin{equation*}
 Y^\nu=\operatorname{span}\{h_I^\nu:I\in\mathcal{A}\}.
\end{equation*}
Since $b>0$ and $b\in L_1^N$, multiplication by $b$ defines an invertible map $U:L_1^N\to L_1^N$, $Ug=bg$, satisfying
\begin{equation*}
 \langle Ug,Uh\rangle_w=\int b^2gh/w=\int gh\,d\nu.
\end{equation*}
We have shown that $(k_I)_{I\in\mathcal{A}}$ is a basis of $Y$ and $k_I/b$ is a nonzero multiple of $h_I^\nu$.
Hence, $U^{-1}(Y)=Y^\nu$ and, since $U$ preserves the inner products, $U^{-1}(Y^{\perp_w})=(Y^\nu)^\perp$.
It follows that $U^{-1}P_wU=T_\nu$, or equivalently,
\begin{equation}\label{eq:conjugate}
 P_wf=b\,T_\nu(f/b),\qquad f\in L_1^N.
\end{equation}
Note that $T_\nu$ is precisely a martingale transform for the standard dyadic filtration, with measure~$\nu/\nu([0,1])$ and predictable multipliers taking values in $\{0,1\}$.
Apply Theorem~\ref{thm:bourgain-davis} with $p=1/2$ to the functions $f_i/b$ and the transform $T_\nu$.
Substituting~\eqref{eq:conjugate} and $d\nu=(b^2/w)\,dt$ cancels the factors $b^2$ on both sides and gives exactly~\eqref{eq:weighted}.
\end{proof}

\section{Proof of the main result}
\label{sec:proof-of-main-result}

We will now prove Theorem~\ref{thm:main} by combining Proposition~\ref{prop:weighted} with Khintchine's inequality and standard inequalities.

\begin{proof}[Proof of Theorem~\ref{thm:main}]
Let $x_1,\dots,x_n\in L_1^N$. For every sign vector $\varepsilon\in \{ -1,1 \}^n$, choose an element $y_{\varepsilon}\in Y$ that minimizes $\|\sum_{j=1}^n \varepsilon_j x_j + y_\varepsilon\|_{L_1}$. Hence, the function
\begin{equation*}
  F_\varepsilon=\sum_{j=1}^n\varepsilon_jx_j+y_\varepsilon
  \qquad
  \text{satisfies}
 \qquad \|F_\varepsilon\|_{L_1}=\biggl\|\sum_{j=1}^n\varepsilon_jx_j+Y\biggr\|_{L_1/Y}.
\end{equation*}
For a fixed real number $\eta > 0$, set
\begin{equation}\label{eq:A-w}
 A=\mathbb{E}_\varepsilon\|F_\varepsilon\|_{L_1}
 \qquad
 \text{and}
 \qquad
 w=\mathbb{E}_\varepsilon|F_\varepsilon|+\eta.
\end{equation}
In particular, $\int w=A+\eta$ and $w>0$. Define the
representatives $z_j=P_wx_j$ for $j=1,\dots,n$. By~\eqref{eq:P},
$P_w$ eliminates all independently chosen vectors $y_\varepsilon$, and we get
\begin{equation*}
  z_j-x_j\in Y
  \qquad
  \text{and}
  \qquad
 P_wF_\varepsilon=\sum_{j=1}^n\varepsilon_jz_j.
\end{equation*}
Proposition~\ref{prop:weighted} applied to the functions~$F_\varepsilon$, $\varepsilon\in \{ -1,1 \}^n$, gives
\begin{equation}\label{eq:bourgain-davis-consequence}
 \int\biggl(\mathbb{E}_\varepsilon
              \Bigl|\sum_{j=1}^n\varepsilon_jz_j\Bigr|^{1/2}\biggr)^4\,\frac{dt}{w}
 \leq B^2\int\bigl(\mathbb{E}_\varepsilon|F_\varepsilon|\bigr)^2\,\frac{dt}{w}
 \leq B^2 A,
\end{equation}
where the last inequality follows from~\eqref{eq:A-w}.
Put $g(t)=\mathbb{E}_\varepsilon|\sum_{j=1}^n\varepsilon_jz_j(t)|^{1/2}$ and apply the Cauchy--Schwarz inequality to $g^2/\sqrt w$ and~$\sqrt w$.
Together with~\eqref{eq:bourgain-davis-consequence} and $\int w = A + \eta$, we obtain
\begin{equation}\label{eq:cauchy-schwarz-estimate}
 \int\biggl(\mathbb{E}_\varepsilon
 \Bigl|\sum_{j=1}^n\varepsilon_jz_j\Bigr|^{1/2}\biggr)^2
 \leq \left(\int g^4/w\right)^{1/2}\left(\int w\right)^{1/2}
 \leq B\sqrt{A(A+\eta)}.
\end{equation}

Finally, using $z_j\in x_j + Y$, Minkowski's integral inequality, Khintchine's inequality (applied pointwise with $p = 1/2$), and~\eqref{eq:cauchy-schwarz-estimate}, we obtain
\begin{align*}
 \biggl(\sum_{j=1}^n \left\|x_j+Y\right\|_{L_1/Y}^2\biggr)^{1/2}
 &\leq\biggl(\sum_{j=1}^n\|z_j\|_1^2\biggr)^{1/2}\\
 &\leq\int\biggl(\sum_{j=1}^n|z_j(t)|^2\biggr)^{1/2}\,dt\\
 &\leq K_{1/2} B\sqrt{A(A+\eta)},
\end{align*}
where $K_{1/2}$ denotes the constant in Khintchine's inequality for $p = 1/2$.
Letting $\eta\downarrow0$ proves~\eqref{eq:main} with constant $C=K_{1/2} B$.
\end{proof}

\section{The Bourgain--Davis inequality}

We present an elementary proof of the martingale transform inequality by Bourgain and Davis~\cite{MR825718} (see also \cite[Section~6.8]{MR3617459}). This proof only uses Burkholder's weak type $(1,1)$ inequality for martingale transforms and standard results from martingale theory, such as Doob's $L_2$ maximal inequality.

Let $(\mathcal F_k)_{k=0}^m$ be a filtration on a probability space $\Omega$
with probability measure $\mathbb{P}$. Assume that $\mathcal F_0$ is trivial
and that each $\mathcal F_k$ is generated by a finite partition of the space into
atoms of positive probability. Write $\mathbb{E}$ for expectation and
$\mathbb{E}_k=\mathbb{E}(\,\cdot\mid\mathcal F_k)$ for conditional expectation.

\begin{theorem}[Bourgain--Davis]\label{thm:bourgain-davis}
Let $0<p<1$, and let $a_1,\ldots,a_m$ be real-valued functions
such that $a_k$ is $\mathcal F_{k-1}$-measurable and $|a_k|\leq1$ for all $k$.
For $f_1,\ldots,f_N\in L_1(\Omega)$, define the martingale transform $T$ with multipliers $(a_k)_{k=1}^m$ by
\begin{equation*}
 Tf_i=\mathbb{E}_0f_i+\sum_{k=1}^m a_k(\mathbb{E}_kf_i-\mathbb{E}_{k-1}f_i).
\end{equation*}
Then there is a constant $C_p$, depending only on $p$, such that
\begin{equation}\label{eq:goal}
 \mathbb{E}\left(\frac1N\sum_{i=1}^N|Tf_i|^p\right)^{2/p}
 \leq C_p^2\, \mathbb{E}\left(\frac1N\sum_{i=1}^N|f_i|\right)^2.
\end{equation}
\end{theorem}

The conclusion can be expressed in mixed-norm notation. To this end, give
$S=\{1,\ldots,N\}$ the uniform probability measure and set
$h(t,i)=f_i(t)$ for $t \in \Omega$ and $i\in S$. Let the operator $T$
act only in the variable $t$, so $(Th)(t,i)=(Tf_i)(t)$.
Then~\eqref{eq:goal} is equivalent to
\begin{equation*}
 \|Th\|_{L_2(\Omega, L_p(S))}
 \leq C_p\|h\|_{L_2(\Omega, L_1(S))}.
\end{equation*}

We will use Burkholder's weak type $(1,1)$ inequality for martingale
transforms, as stated in~\cite[Section~5.1, inequality (5.2)]{MR3617459}:
If $(u_k)_{k=0}^{m}$ is a real martingale with $u_0=0$ and
$v_k=\sum_{l=1}^k b_l(u_l-u_{l-1})$, where $(b_l)$ is predictable
and $|b_l|\leq1$, then
\begin{equation*}
 \mathbb{P}\{v^*>t\}\leq\frac{C}{t}\mathbb{E}|u_m|\quad(t>0),
 \qquad v^*=\max_{0\leq k\leq m}|v_k|,
\end{equation*}
with an absolute constant $C$. Consequently, for a constant $K_p$
depending only on $p$,
\begin{equation}\label{eq:moment}
 \mathbb{E}(v^*)^p\leq K_p(\mathbb{E}|u_m|)^p.
\end{equation}

\begin{proof}[Proof of Theorem~\ref{thm:bourgain-davis}]
For $1\leq i\leq N$ and $0\leq k\leq m$, put
\begin{equation*}
 f_{i,k}=\mathbb{E}_{k}f_i,\qquad
 g_{i,k}=f_{i,0}+\sum_{l=1}^k a_l(f_{i,l}-f_{i,l-1}),
\end{equation*}
and define
\begin{equation*}
 F=\frac1N\sum_{i=1}^N|f_i|,\qquad
 F_k=\mathbb{E}_{k}F,\qquad F^*=\max_{0\leq k\leq m}F_k,
\end{equation*}
\begin{equation*}
 G_k=\frac1N\sum_{i=1}^N|g_{i,k}|^p,\qquad
 G^*=\max_{0\leq k\leq m}G_k.
\end{equation*}
It suffices to prove $\mathbb{E}(G^*)^{2/p}\le C_p \mathbb{E} F^2$, since $g_{i,m}=Tf_i$ for all~$i$.

\medskip\noindent\textit{Step 1: estimate the transforms after a fixed time.}
For $0\leq s\leq m$, set
\begin{equation*}
 Z_s=\frac1N\sum_{i=1}^N\max_{s\leq k\leq m}|g_{i,k}-g_{i,s}|^p.
\end{equation*}
We claim that, for a constant $B_p$ depending only on $p$,
\begin{equation}\label{eq:conditional}
 \mathbb{E}_{s} Z_s\leq B_pF_s^p.
\end{equation}
On an atom $A$ of $\mathcal F_s$, use the probability space
$(A,\mathcal F|_A,\mathbb{P}_A)$, where $\mathbb{P}_A(E)=\mathbb{P}(E)/\mathbb{P}(A)$ for
measurable $E\subseteq A$, with filtration
$(\mathcal F_{s+l}|_A)_{l=0}^{m-s}$. Expectation under $\mathbb{P}_A$
is the value of $\mathbb{E}_{s}$ on~$A$.
We will apply~\eqref{eq:moment} on this probability space to the martingale
$(f_{i,s+l}-f_{i,s})_{l=0}^{m-s}$ restricted to $A$. Its transform with
respect to the predictable process
$(a_{s+l}|_A)_{l=1}^{m-s}$ is
$(g_{i,s+l}-g_{i,s})_{l=0}^{m-s}$ restricted to $A$.
Therefore, using~\eqref{eq:moment} and the inequality
\begin{equation*}
 \mathbb{E}_{s}|f_i-f_{i,s}|
 \leq\mathbb{E}_{s}|f_i|+|f_{i,s}|
 \leq2\mathbb{E}_{s}|f_i|,
\end{equation*}
we obtain
\begin{equation*}
 \mathbb{E}_s\max_{s\leq k\leq m}|g_{i,k}-g_{i,s}|^p
 \leq K_p(\mathbb{E}_s|f_i-f_{i,s}|)^p
 \leq 2^pK_p(\mathbb{E}_s|f_i|)^p,
 \qquad 1\leq i\leq N,
\end{equation*}
where $K_p$ depends only on $p$.
Averaging these estimates gives
\begin{align*}
 \mathbb{E}_{s} Z_s
 &\leq 2^pK_p\frac1N\sum_{i=1}^N(\mathbb{E}_{s}|f_i|)^p\\
 &\leq 2^pK_p\left(\frac1N\sum_{i=1}^N\mathbb{E}_{s}|f_i|\right)^p
 =2^pK_pF_s^p.
\end{align*}
This proves~\eqref{eq:conditional}.

\medskip\noindent\textit{Step 2: a threshold estimate for $G^{*}$.}
Fix $\lambda>0$ and define the stopping time $\sigma(t)$ to be the smallest $k\in \{ 0,\dots,m \}$ such that $G_k(t) > \lambda$, or $\infty$ if no such $k$ exists. Then $\{\sigma\leq m\}=\{G^*>\lambda\}$. For $1\leq k\leq m$, using $|a_k|\leq1$ and $|f_{i,k}|\le \mathbb{E}_k|f_i|$,  we obtain
\begin{align*}
 G_k
 &\leq G_{k-1}+\frac1N\sum_{i=1}^N|g_{i,k}-g_{i,k-1}|^p\\
 &\leq G_{k-1}
       +\left(\frac1N\sum_{i=1}^N|f_{i,k}-f_{i,k-1}|\right)^p\\
 &\leq G_{k-1}+(F_k+F_{k-1})^p.
\end{align*}
Hence, on $\{1\leq\sigma\leq m\}$, we have
\begin{equation}\label{eq:1}
  G_\sigma\leq \lambda+2^p(F^*)^p.
\end{equation}
The same bound holds on $\{ \sigma=0 \}$,
because $G_0\leq F_0^p\le (F^{*})^p$ by the concavity of $x\mapsto x^p$ and $|f_{i,0}|\le \mathbb{E}_0|f_i|$.

Now let $0\le k\le m$. On the set $\{ \sigma > k \}$, we have $G_k\leq \lambda$, and on $\{ \sigma\le k \}$, subadditivity of $x\mapsto x^p$ and~\eqref{eq:1} give $G_k\leq G_\sigma+Z_\sigma \le \lambda + 2^p(F^{*})^p + Z_\sigma$. 
Consequently,
\begin{equation}\label{eq:2}
 (G^*-\lambda)_+
 \leq\mathbf{1}_{\{\sigma\leq m\}}\bigl(2^p(F^*)^p+Z_\sigma\bigr),
\end{equation}
where $x_+=\max\{x,0\}$. Here $Z_\sigma$ is set equal to zero when $\sigma=\infty$.
Using~\eqref{eq:conditional}, we get
\begin{align*}
 \mathbb{E}[\mathbf{1}_{\{\sigma\leq m\}}Z_\sigma]
 &=\sum_{s=0}^m\mathbb{E}\mathbb{E}_s[\mathbf{1}_{\{\sigma=s\}} Z_s]
 =\sum_{s=0}^m\mathbb{E}[\mathbf{1}_{\{\sigma=s\}}\mathbb{E}_s Z_s]\\
 &\leq B_p\sum_{s=0}^m\mathbb{E}[\mathbf{1}_{\{\sigma=s\}}F_s^p]
   \leq B_p\mathbb{E}[\mathbf{1}_{\{G^*>\lambda\}}(F^*)^p].
\end{align*}
Thus, with $D_p=2^p + B_p$ and $\{ \sigma \le m \} = \{ G^{*} > \lambda \}$, from~\eqref{eq:2} we obtain
\begin{equation}\label{eq:excess}
 \mathbb{E}(G^*-\lambda)_+
 \leq D_p\mathbb{E}[\mathbf{1}_{\{G^*>\lambda\}}(F^*)^p].
\end{equation}

\medskip\noindent\textit{Step 3: integrate the estimate.}
Put $r=2/p$. Multiply~\eqref{eq:excess} by $r(r-1)\lambda^{r-2}$ and
integrate over $\lambda>0$. For every $h\geq0$,
\begin{align*}
 r(r-1)\int_0^\infty \lambda^{r-2}(h-\lambda)_+\,d\lambda&=h^r,\\
 r(r-1)\int_0^\infty \lambda^{r-2}\mathbf{1}_{\{h>\lambda\}}\,d\lambda&=rh^{r-1}.
\end{align*}
It follows that
\begin{align*}
 \mathbb{E}(G^*)^r
 &\leq rD_p\mathbb{E}[(F^*)^p(G^*)^{r-1}]\\
 &\leq rD_p\bigl(\mathbb{E}(F^*)^2\bigr)^{1/r}
                \bigl(\mathbb{E}(G^*)^r\bigr)^{(r-1)/r},
\end{align*}
where the last step holds by Hölder's inequality with exponents $r$ and $r/(r-1)$.
If $\mathbb{E}(G^*)^r=0$, the conclusion is immediate. Otherwise division
and raising to the power $r$ gives
\begin{equation*}
 \mathbb{E}(G^*)^{2/p}\leq(rD_p)^{2/p}\mathbb{E}(F^*)^2.
\end{equation*}
Doob's $L_2$ maximal inequality for the nonnegative martingale
$(F_k)$ gives $\mathbb{E}(F^*)^2\leq4\mathbb{E} F^2$. This proves~\eqref{eq:goal}
with a constant depending only on $p$.
\end{proof}

\bibliographystyle{abbrv}%
\bibliography{bibliography}%

\end{document}